\documentclass[12pt]{amsart}
\usepackage[top=1in, bottom=1in, left=1in, right=1in]{geometry}
\usepackage{amsfonts}
\usepackage{amsmath}
\usepackage{comment}
\usepackage{float}
\usepackage{amssymb}
\usepackage{graphicx}
\usepackage{bbm}
\usepackage{comment}
\usepackage{mathrsfs}
\numberwithin{equation}{section}
\usepackage{times}
\usepackage[backend=biber, sorting=nyt, url=false,
maxnames = 100,
doi=false]{biblatex}
\usepackage{siunitx}
\usepackage[usenames,dvipsnames]{color}
\usepackage{comment}
\usepackage{xcolor}
\usepackage{mathtools}
\usepackage{bm}
\usepackage{esvect}
\usepackage{hyperref}

\hypersetup{
    colorlinks = true,
linkcolor={black},
urlcolor={blue},
citecolor={blue},    
urlcolor = {blue},
citebordercolor = {0.33 .58 0.33},
 linkbordercolor = {0.99 .28 0.23},
 breaklinks=true}
 
\newcommand{\F}{\mathbb{F}}
\newcommand{\R}{\mathbb{R}}
\newcommand{\Q}{\mathbb{Q}}
\newcommand{\N}{\mathbb{N}}
\newcommand{\Z}{\mathbb{Z}}

\newtheorem{thm}{Theorem}[section]
\newtheorem{defn}[thm]{Definition}

\newtheorem{lem}[thm]{Lemma}
\newtheorem{rem}[thm]{Remark}
\newtheorem{conj}[thm]{Conjecture}

\theoremstyle{remark}

\usepackage{geometry}
\title{Admissible Pairs: A Variation to Pollock Conjectures}
\author{Anji Dong, Vi Anh Nguyen, Alexandru Zaharescu}

\address{
Anji Dong: Department of Mathematics,
University of Illinois Urbana-Champaign,
Altgeld Hall, 1409 W. Green Street,
Urbana, IL, 61801, USA}
\email{anjid2@illinois.edu}

\address{
Vi Anh Nguyen: Department of Mathematics,
University of Illinois Urbana-Champaign,
Altgeld Hall, 1409 W. Green Street,
Urbana, IL, 61801, USA}
\email{vianhan2@illinois.edu}

\address{
Alexandru Zaharescu: Department of Mathematics,
University of Illinois Urbana-Champaign,
Altgeld Hall, 1409 W. Green Street,
Urbana, IL, 61801, USA and ``Simion Stoilow" Institute of Mathematics of the Romanian Academy, 
P. O. Box 1-764, RO-014700 Bucharest, Romania}
\email{zaharesc@illinois.edu}  
\begin{document}
\nocite{*}
\setcounter{tocdepth}{1}
\begin{abstract}
We introduce a notion of an admissible pair, and prove a variation to Pollock's conjectures on icosahedral and dodecahedral numbers. 
\end{abstract}
\maketitle

\setcounter{tocdepth}{-1}
\section{Introduction}\label{sec: Introduction}
In 1638, Fermat proposed the famous Fermat polygonal number theorem, which states that every positive integer is a sum of at most $n$ $n$-gonal numbers. A major extension of this result is the higher-dimensional generalization known as Waring's problem, which asks whether, for each $k\in\N$, there exists a natural number $g(k)$ such that each positive integer is the sum of at most $g(k)$ $k\text{-th}$ powers. Hilbert \cite{hilbert1909} proved the existence of \( g(k) \) for every \( k \in \mathbb{N} \). 
Later, Wieferich \cite{Wieferich1908BeweisDS} and Kempner \cite{Kempner1912BemerkungenZW} 
established that \( g(3) = 9 \). The case \( g(4) = 19 \) was resolved by 
Balasubramanian \cite{balasubramanian}, Deshouillers, and Dress \cite{deshouillers}. Chen \cite{chen1964} demonstrated that \( g(5) = 37 \), while Pillai \cite{pillai1940} 
showed that \( g(6) = 73 \). 

The study of perfect powers in Waring’s problem naturally extends to polyhedral
numbers. Let $f$ be a polyhedral polynomial---a polynomial over $\Q$ with a positive leading coefficient, satisfying $f(0)=0$ and $f(1)=1$. Analogous to Waring’s problem, one may ask: What is the least number $g(f_n)$ such that every positive integer can be written as a sum of at most $g(f_n)$ values of $f_n$? In the 1850s, Sir Frederick Pollock \cite{pollock} proposed conjectures for polyhedral numbers associated with the five Platonic solids. To be precise, let $T_n, O_n,C_n,I_n$, and $D_n$ denote the $n\text{-th}$ tetrahedral, octahedral, cubical, icosahedral, and dodecahedral numbers, respectively. Pollock conjectured:
\[
g(T_n)=5,\ g(O_n)=7,\ g(C_n)=9,\  g(I_n)=13,\  g(D_n)=21.
\]

It was later found that $95$ is a counterexample to his icosahedral conjecture, where $95$ requires $15$ terms. Similarly, $79$ is a counterexample to the dodecahedral case, which requires $22$ terms. The conjecture for cubes was mentioned above. In 2016, Brady \cite{Brady_2015} proved Pollock's octahedral number conjecture for all integers greater than $e^{10^7}$. More recently, Basak, Saettone, and two of the authors \cite{DoZa2024} confirmed the corrected icosahedral and dodecahedral conjectures, which state that $g(I_n)=15$ and $g(D_n)=22$.

Pollock's conjectures concern representations of integers using polyhedral numbers of a single form. In this paper, we generalize the definition of $g(f_n)$ to representations involving two forms. 

\begin{defn}
   Let $f$ and $h$ be polyhedral polynomials. A pair $(u,v)\in\Z_{\geq 0}\times \Z_{\geq 0}$ is called an \textbf{admissible pair} associated with $f$ and $h$ if every positive integer $m$ can be expressed as
   \[
   m = \sum_{i=1}^a f(n_i) + \sum_{j=1}^b h(\ell_j),
   \]
   for some $a \leq u$, $b \leq v$, and $n_i, \ell_j \in \mathbb{N}\cup\{0\}$.
\end{defn}
 Immediately, we see that the admissible pairs form an upward-closed set: if $(u,v)$ is admissible, then so is $(u',v')\in\Z^2$ for any $u'\geqslant u$ and $v'\geqslant v$. 
 \begin{defn}
Let $f$ and $h$ be polyhedral polynomials. The set $\mathbf{g(f_n, h_n)}$ consists of all minimal admissible pairs $(u,v)$ in the sense that there exists no admissible pair $(u',v')$ with $u' \leq u$, $v' \leq v$, and at least one inequality is strict. We call any $(u,v)\in g(f_n, h_n)$ an \textbf{essential pair}.
\end{defn}

In this paper, our goal is to study $g(I_n, D_n)$, where $I_n$ and $D_n$ denote the $n$-th icosahedral and dodecahedral numbers respectively, defined by
\begin{align*}
I_n &= \frac{n(5n^2-5n+2)}{2}, 
\ \textrm{and} \quad D_n = \frac{n(3n-1)(3n-2)}{2},
\end{align*} 
for $n \in \mathbb{N}\cup\{0\}$. The first three icosahedral numbers are $0,1,$ and $12$. Similarly, the first three dodecahedral numbers are $0,1,$ and $20$. To simplify notation, we will refer to essential and admissible pairs associated specifically with the icosahedral ($I_n$) and dodecahedral ($D_n$) polynomials simply as \textbf{essential pairs} and \textbf{admissible pairs}. The methods developed here extend naturally to other pairs of polyhedral polynomials to identify possible admissible pairs. 

By \cite{DoZa2024}, both $(15,0)$ and $(0,22)$ are essential pairs. Consequently, all pairs $(u,v)\in\Z_{\geq 0}\times \Z_{\geq 0}$ such that $u\geqslant 15$ or $v\geqslant 22$ are admissible. This leaves us a bounded grid of lattice points to be considered for possible remaining essential pairs. Upon checking the first $12000$ natural numbers for admissible pairs using Algorithm I in \cite{github}, we conjecture the following:
\begin{conj}\label{conj: main conjecture}
    The set of all essential pairs is 
\begin{align*}
 g(I_n, D_n)=\{&(5,6),(4,7),(3,8),(2,9), (1,11),(0,22),(6,5),(7,4),(8,3),(9,2),(10,1),\\
 &(15,0)\}.   
\end{align*}
\end{conj}
\begin{rem}
$(1,10)$ is not an admissible pair, and the smallest counterexample is $31$.
\end{rem}

Our main result confirms Conjecture \ref{conj: main conjecture}:
\begin{thm}\label{thm: main theorem}
   The set of all essential pairs is  
\begin{align*}
 g(I_n, D_n)=\{&(5,6),(4,7),(3,8),(2,9), (1,11),(0,22),(6,5),(7,4),(8,3),(9,2),(10,1),\\
 &(15,0)\}.   
\end{align*}
    Figure \ref{fig:pairs} shows the complete solutions for essential (red) and admissible (blue) pairs. 
\end{thm}

\begin{figure}[ht]
    \centering
    \includegraphics[width=0.6\linewidth]{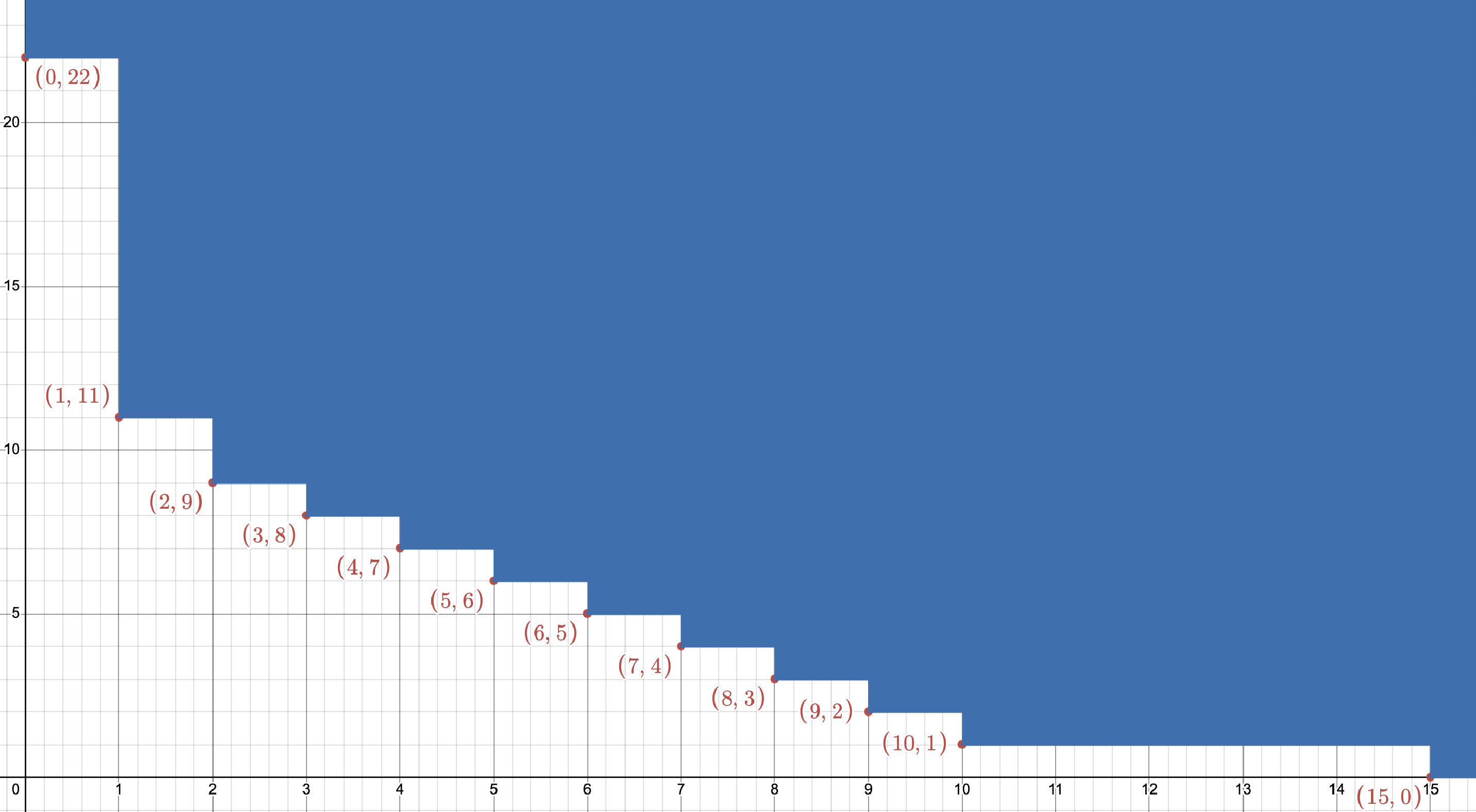}
    \caption{Complete solution for essential and admissible pairs.}
    \label{fig:pairs}
\end{figure}

\subsection*{Structure of the Paper} The paper is organized as follows. In Section \ref{sec: 6+2 lemmas}, we explore the basic properties of an admissible pair and establish two auxiliary lemmas needed in the proof of Theorem \ref{thm: main theorem}. To this end, the lemmas reduce our initial problem to finding essential pairs for a finite set of natural numbers. In Section \ref{sec: proof of main theorem}, we introduce an explicit algorithm, which in turn leads to a complete proof of Theorem \ref{thm: main theorem}.  

\section{"Six Plus Two" Lemmas}\label{sec: 6+2 lemmas}
In this section, we will prove two auxiliary lemmas in preparation for the proof of the main theorem. To begin with, recall the first couple of icosahedral and dodecahedral numbers. Note that for any $(u,v)\in\Z_{\geq 0}\times \Z_{\geq 0}$ to be an admissible pair, we must have $u+v\geqslant 11$ in order to represent the number $11$. This implies that either $u$ or $v$ needs to be greater than or equal to $6$. Therefore, our first strategy is to represent any large enough integer as a sum of eight numbers from the set of icosahedral and dodecahedral numbers. This leaves us with a finite set of positive integers to check for any possible admissible pair. 

To be precise, we will prove the following two lemmata.
\begin{lem}\label{lem: 6 I + 2 D}
    Any integer $m\geqslant \num{1.0755e31}$ can be represented as a sum of six icosahedral numbers and two dodecahedral numbers. 
\end{lem}

\begin{lem} \label{lem: 6 D + 2 I}
     Any integer $m\geqslant \num{2.9446e33}$ can be represented as a sum of two icosahedral numbers and six dodecahedral numbers. 
\end{lem}

Besides these two lemmata, the proof of the main theorem also hinges on two other useful lemmata in \cite{DoZa2024}. 
\begin{lem}[{\cite[Theorem~9.3]{DoZa2024}}] \label{lem: 8I}
    Any $m\in\N$ with $m\geqslant 10^{36}$ is a sum of 8 icosahedral numbers.
\end{lem}
\begin{lem}[{\cite[Theorem~9.6]{DoZa2024}}] \label{lem: 8D}
    Any $m\in\N$ with $m\geqslant 10^{39}$ is a sum of 8 dodecahedral numbers.
\end{lem}
\begin{rem}
    At first glance, it seems that we also need lemmas on "Seven plus One" and "One plus Seven". The reason the above four lemmata are sufficient for the proof of Theorem \ref{thm: main theorem} is that any admissible pair $(u,v)$ will require $u+v\geqslant 11$. So, if one of $u$ or $v$ is $7$, the other must be larger than $3$, which means that we can always manually pick out "Six plus Two" or "Two plus Six" from any possible admissible pairs. 
\end{rem}

\begin{rem}
The method we employed to determine the complete set of essential pairs for icosahedral and dodecahedral polynomials requires each pair to contain at least eight terms. As a result, this method is only applicable for finding essential pairs of polyhedral polynomials $f$ and $h$ that satisfy $g(f), g(h) \geq 8$. For instance, when applied to the tetrahedral polynomial $T_n$---where Pollock conjectured that $g(T_n) = 5$---the method fails.
\end{rem}

The proof of Lemma \ref{lem: 6 D + 2 I} is analogous to that of Lemma \ref{lem: 6 I + 2 D}, and thus we will present the full proof of Lemma \ref{lem: 6 I + 2 D}. To begin with, we first consider the sum of six icosahedral numbers using Linnik's method. Let $f(x) = \frac{5x^3}{2}-\frac{5x^2}{2}+x,$ and write
\begin{align}\label{Six Icosahedral Numbers}
\mathcal{S}(6,\mathcal{I}) &= f(x + a + 1) +f(x - a) + f(x+b+1) + f(x- b) + f(x + c + 1) + f(x - c)\notag
\\
&= \frac{3}{2}(10x^{3} + 5x^{2} + 9x + 2) + (30x + 5)\left( \frac{c(c+1)}{2} + \frac{b(b+1)}{2} + \frac{a(a+1)}{2} \right),
\end{align}
with $a,b,c <x$. Since $\frac{n(n+1)}{2}$ is the general form of the $n$-th triangular number, we see that 
\[
 \frac{c(c+1)}{2} + \frac{b(b+1)}{2} + \frac{a(a+1)}{2} 
\]
is the sum of three triangular numbers. Due to Gauss, we know that every positive integer is the sum of at most three triangular numbers. We use this to prove that for any large and fixed $m$, there exist positive integers $\ell$ and $r$ such that 
\begin{align}\label{Eight Icosahedral Numbers}
\mathcal{S}(6,\mathcal{I}) + g(\ell) + g(r)= m,
\end{align}
where $g(x) = \frac{9x^3}{2}-\frac{9x^2}{2}+x$.

From \eqref{Six Icosahedral Numbers} and \eqref{Eight Icosahedral Numbers}, it follows that 
\[
\left( \frac{c(c+1)}{2} + \frac{b(b+1)}{2} + \frac{a(a+1)}{2} \right) = \frac{m - g(\ell) - g(r) - \frac{3}{2}(10x^{3} + 5x^{2} + 9x + 2)}{30x + 5}.
\]

This implies that the expression
\[
m - g(\ell) - g(r) - \frac{3}{2}(10x^{3} + 5x^{2} + 9x + 2)
\]
must be divisible by $30x + 5$. Thus, the problem reduces to the following. For any positive integer $m$ (maybe sufficiently large), it suffices to find $x,\ell,r$ in $\mathbb{Z}^+$ satisfying:
\begin{align*}
   &1. \quad 0< L - g(\ell) - g(r) < 45x^3\\
   &\textrm{and}\\
   &2. \quad (30x+5)\mid (L - g(\ell) - g(r)),
\end{align*}
where $L = m - \frac{3}{2} \left( 10x^{3} + 5x^{2} + 9x + 2 \right) \in \Z^+$. We will refer to the above conditions as Condition 1 and Condition 2 respectively. 

In light of the idea of counting Lehmer points with respect to an irreducible algebraic curve over a finite field, as well as with a restricted range of values and congruence relations discussed by Cobeli and Zaharescu \cite{cobeli}, our strategy now is to simplify the two conditions in order to fit the setting of Lehmer points. 

We first consider Condition 1, and our goal is to find a range of $\ell$ and $r$ that will satisfy the condition. Fix $t,\delta \in \R^{+}$. Our strategy is to find a prime $p$ such that $p \equiv 1 \bmod 6$ and 
\begin{align}
    (t - \delta)p^{3} \leqslant m \leqslant (t + \delta)p^{3}.\label{m bound}
\end{align}

This means that $p$ needs to satisfy the inequalities
\begin{align}\label{Primes in AP Condition}
\bigg ( \frac{m}{t + \delta} \bigg)^{1/3} \leqslant p \leqslant \bigg ( \frac{m}{t - \delta} \bigg)^{1/3}.
\end{align}
If such a $p$ is guaranteed, then we let $x = (p-1)/6$. Therefore, we have that $30x + 5 = 5p.$ Now substituting $x$ into \eqref{m bound}, we find that
\begin{align*}
    &(t-\delta)p^3 \leqslant L +\frac{5(p-1)^3}{72}+\frac{15(p-1)^2}{72}+\frac{9(p-1)}{4}+3\leqslant (t+\delta)p^3.
\end{align*}

When $p$ is sufficiently large, we may reduce this to
\begin{align}\label{L Inequality}
\bigg (t-2\delta-\frac{5}{72}\bigg )p^3\leqslant L\leqslant \bigg (t+\delta-\frac{5}{72}\bigg)p^3.
\end{align}

Therefore, Condition 1 will be satisfied if
\begin{align}
L - \frac{5}{24}p^{3} < g(\ell) + g(r) < L.\label{alternative form of condition 1}  
\end{align}

Recalling the range of $L$, we must choose $r$ and $\ell$ such that
\begin{align}
   \left( t - 2\delta - \frac{5}{18} \right) p^{3} < g(\ell) + g(r) < \left( t + \delta - \frac{5}{72}\right) p^{3}. \label{f(l)+f(r) bound 2}
\end{align}

If $r, \ell \in \left( (\alpha - \gamma) p, (\alpha + \gamma) p\right)$ for some $\alpha,\gamma>0$, then
\begin{align}
  g(\ell) + g(r) < 9(\alpha + \gamma)^{3}p^{3}.  \label{f(l)+f(r) upper bound}
\end{align}

Combining this with \eqref{f(l)+f(r) bound 2}, we may assume 
\begin{align*}
(\alpha + \gamma)^{3} < \frac{t}{9} - \frac{5}{648} + \frac{\delta}{9},
\end{align*}
which implies that
\begin{align}\label{Alpha Range Lower}
\alpha < \left(  \frac{t}{9} - \frac{5}{648} + \frac{\delta}{9} \right)^{1/3} - \gamma.
\end{align}

Similarly, we need 
\begin{align}\label{flfr lower bound}
    9(\alpha - \gamma)^{3}p^{3}<g(\ell) + g(r),
\end{align}
so we may assume that
\begin{align}\label{Alpha Range Upper}
\alpha > \left( \frac{t}{9} - \frac{2\delta}{9} - \frac{5}{162} \right)^{1/3} + \gamma.
\end{align}

Thus, Condition 1 is reduced to finding $\ell,r\in\Z^+$ such that $r,\ell\in \left( (\alpha - \gamma) p, (\alpha + \gamma) p\right)$, where $\alpha$ and $\gamma$ are positive real numbers satisfying \eqref{Alpha Range Lower} and \eqref{Alpha Range Upper}.

Now we consider Condition 2, which can be viewed as a congruence condition. To this end, we need to find $\ell$ and $r$ such that 
$g(\ell) + g(r) \equiv L \bmod 5p.$ Using the Chinese Remainder Theorem, it suffices to consider the congruence equations
\[
g(\ell)+g(r) \equiv L \bmod 5 \text{ and } g(\ell) + g(r) \equiv L \bmod p.
\]

The first equivalence condition is the same as
\begin{align}
    2\ell^3+2r^3-2\ell^2-2r^2+\ell+r\equiv L \bmod 5.\label{first congruence}
\end{align}

Consider $\ell_0,r_0\in \Z/5\Z$ such that $\ell\equiv \ell_0\bmod 5$ and $r\equiv r_0\bmod 5$.  We divide \eqref{first congruence} into five cases: if $L\equiv 0\bmod 5$, then let $\ell_0 =r_0 = 0$. If $L\equiv 1\bmod 5$, then let $\ell_0 =1$ and $r_0=2$. If $L\equiv 2\bmod 5$, then let $\ell_0 =r_0=1$. If $L\equiv 3\bmod 5$, then let $\ell_0=r_0=3$. If $L\equiv 4\bmod 5$, then let $\ell_0=2$ and $r_0=3$. In any case, the chosen $\ell_0$ and $r_0$ will satisfy \eqref{first congruence}.  

Therefore, after reducing both Condition 1 and 2, our original problem will be solved if we can find $\ell,r\in\Z^+$ such that
\begin{align}
r,\ell &\in \left( (\alpha - \gamma) p,(\alpha + \gamma) p\right),\notag\\
\ell &\equiv \ell_0\bmod 5,
r\equiv r_0\bmod 5,
\textrm{and } g(\ell)+g(r)\equiv L\bmod p,\label{final conditions}
\end{align}
where $\alpha$ satisfies the inequalities \eqref{Alpha Range Lower} and \eqref{Alpha Range Upper}. This falls exactly in the setting of Lehmer points mentioned earlier, so we first construct the setup accordingly. Let $1 \leqslant L_{0} \leqslant 5$ be such that $L \equiv L_{0} \bmod 5$. Let $\mathcal{C}$ denote the algebraic curve over $\mathbb{F}_{p}$ given by 
\begin{align}\label{Curve definition}
\mathcal{C}: g(x_1)+g(x_2)-L=0,
\end{align}
where $\F_p$ refers to the prime field of characteristic $p$. 

Define the vectors
\begin{align}
\textbf{x}& := \begin{bmatrix}
           x_1 \\
           x_2 \\
         \end{bmatrix},
    \textbf{a} := \begin{bmatrix}
           5 \\
           5 \\
         \end{bmatrix},
    \textbf{b} := \begin{bmatrix}
            \ell_0\\
            r_0\\
         \end{bmatrix},
    \textbf{t}:= \begin{bmatrix}
         t_1 \\
          t_2 \\
         \end{bmatrix},
         \textbf{g}:= \begin{bmatrix}
         g_1 \\
          g_2 \\
         \end{bmatrix}. \label{vectors}
\end{align}

The set of Lehmer points $\mathcal{L}(p,\mathcal{C},\textbf{a},\textbf{b})$ with respect to the prime $p$, the curve $\mathcal{C}$, and the vectors $\textbf{a},\textbf{b}$ is defined as
\begin{align}\label{Lehmer Point Definition}
\mathcal{L}(p,\mathcal{C},\textbf{a},\textbf{b}) := \{\textbf{x} : x_1 \equiv b_1\bmod a_1, x_2\equiv b_2 \bmod a_2, (x_1,x_2)\in\mathcal{C}, 0\leqslant x_1, x_2<p\} . 
\end{align}

For the prime $p$ and the vectors $\textbf{t}$ and $\textbf{g}$, consider the intervals $U_j = ((t_j-g_j)p, (t_j+g_j)p)$ for $j \in \{1,2\}$. Then the distribution function of the Lehmer points is given by
\begin{align}\label{Lehmer Distribution function}
F(p,\mathcal{C},\textbf{a},\textbf{b};\textbf{g},\textbf{t}):=\#\{\textbf{x}\in \mathcal{L}(p,\mathcal{C},\textbf{a},\textbf{b}):  x_j\in U_j, 0\leqslant t_j\pm g_j\leqslant 1, j=1,2 \}.
\end{align}

With the above setup, we have the following lemma.
\begin{lem}[\cite{DoZa2024}, Lemma 9.4] \label{lem: ZC theorem}
Let $p$ be a prime. Suppose $\mathcal{C}$, $\textnormal{\textbf{a}, \textbf{b}, \textbf{g}, \textbf{t}}$ and $ F$ be as in \eqref{vectors}--\eqref{Lehmer Distribution function}. Then for $p\geqslant 10^{10}$,
\begin{align*}
F(p,\mathcal{C},\textnormal{\textbf{a}, \textbf{b}; \textbf{g}, \textbf{t}}) = \frac{4g_1g_2}{25}(p+1)+ \mathcal{E},
\end{align*}
where $|\mathcal{E}| \leqslant 6.0003\sqrt{p}\log^2 p$.
\end{lem}

Now we are ready to prove Lemma \ref{lem: 6 I + 2 D}.

\noindent {\it Proof of Lemma \ref{lem: 6 I + 2 D}.\quad}
We first choose $t = 0.55, \gamma = 0.14$, and $\delta = 0.3$. Then we choose $\alpha =0.14 $, which satisfies the desired range given in \eqref{Alpha Range Lower} and \eqref{Alpha Range Upper}. Apply Lemma \ref{lem: ZC theorem} with $t_1=t_2=\alpha$ and $g_1=g_2=\gamma$. In order to have $F(p,\mathcal{C},\textnormal{\textbf{a}, \textbf{b}; \textbf{g}, \textbf{t}}) > 0$, it suffices to achieve
\begin{align}\label{Prime Requirement}
\gamma^2(p+1)>\frac{25}{4} (6.0003\sqrt{p}\log^2p).
\end{align}

When $m\geqslant \num{1.0755e31}$, we have $p> \num{2.33e10}$, which satisfies \eqref{Prime Requirement}. Now the proof will be fulfilled once we show that when $m\geqslant \num{1.0755e31}$, there exists a prime $p\equiv 1\bmod 6$ such that 
\[
\frac{1}{\sqrt[3]{0.85}}\sqrt[3]{m}\leqslant p\leqslant\frac{1}{\sqrt[3]{0.25}}\sqrt[3]{m},
\]
and Condition 1 is satisfied with such choice of $p$. To do so, we first restrict the range of $p$ to be the sub-interval $[1.536\sqrt[3]{m}, 1.552\sqrt[3]{m}]$. This range of $p$, when $m\geqslant 73073$, guarantees us Condition 1. Next, to show the existence of $p$ in this sub-interval, we apply the result by Ramaré and Rumely \cite[Theorem 1]{ramare1996} with $\varepsilon = 0.004560$. To this end, we obtain 
\begin{align*}
    \max_{1\leqslant y\leqslant 1.536\sqrt[3]{m}}\bigg|\theta(y;6,1)-\frac{y}{2}\bigg|\leqslant 0.00350208\sqrt[3]{m},
\end{align*}
where $\theta(y;k,a)$ is the Chebyshev function. Taking $y=1.536\sqrt[3]{m}$, we have
\begin{align}\label{lower chebyshev}
    \bigg|\theta(1.536\sqrt[3]{m};6,1)-\frac{1.536\sqrt[3]{m}}{2}\bigg|\leqslant 0.00350208\sqrt[3]{m}.
\end{align}

Similarly, take $y=1.552\sqrt[3]{m}$, we have
\begin{align}\label{upper chebyshev}
    \bigg|\theta(1.552\sqrt[3]{m};6,1)-\frac{1.552\sqrt[3]{m}}{2}\bigg|\leqslant 0.00353856\sqrt[3]{m}.
\end{align}

Combining \eqref{lower chebyshev} and \eqref{upper chebyshev}, we obtain
\begin{align*}
    \sum\limits_{\substack{1.536\sqrt[3]{m}<p\leqslant 1.552\sqrt[3]{m} \\ p \equiv 1 \bmod 6}}\log p&=\theta(1.552\sqrt[3]{m};6,1)-\theta(1.536\sqrt[3]{m};6,1)\notag\\
    &\geqslant (0.776\sqrt[3]{m}-0.00353856\sqrt[3]{m})-(0.768\sqrt[3]{m}+0.00350208\sqrt[3]{m})\notag\\
    & = 0.00095936\sqrt[3]{m} > 0.
\end{align*}

Hence, the existence of such prime $p$ is guaranteed in the sub-interval $[1.536\sqrt[3]{m}, 1.552\sqrt[3]{m}\ ]$. This completes the proof of Lemma \ref{lem: 6 I + 2 D}. \qed

We can follow the same method to prove Lemma \ref{lem: 6 D + 2 I}, replacing icosahedral sums with dodecahedral sums.

\noindent {\it Proof of Lemma \ref{lem: 6 D + 2 I}.\quad}
First, we have the sum of six dodecahedral numbers 
\begin{align}\label{Six Dodecahedral Numbers}
\mathcal{S}(6,\mathcal{D}) &= g(x+a+1) + g(x-a) + g(x+b+1) + g(x-b) + g(x+c+1) + g(x-c)\notag
\\
&= \frac{3}{2}(18x^3 + 9x^2 + 13x+2) + (54x+9)\left(\frac{c(c+1)}{2} + \frac{b(b+1)}{2} + \frac{a(a+1)}{2}\right).
\end{align}
where $g(x) = \frac{9x^3}{2} - \frac{9x^2}{2} + x$. Similar to above, we use this to prove that for any large and fixed $m$, there exist positive integers $l$ and $r$ such that
\begin{align*}
\mathcal{S}(6,\mathcal{D}) + f(\ell) + f(r)= m.
\end{align*}

The problem is reduced as follows. For any large enough integer $m$, it suffices to find $x,\ell,r$ in $\mathbb{Z}^+$ satisfying:
\begin{align*}
   &1. \quad 0< L - f(\ell) - f(r) < 
   81x^3\\
   &\textrm{and}\\
   &2. \quad (54x+9)\mid (L - f(\ell) - f(r)),
\end{align*}
where $L = m - \frac{3}{2} \left( 18x^{3} + 9x^{2} + 13x + 2 \right) \in \Z^+$. We will refer to the above conditions as Condition 1 and Condition 2 respectively. 

Fix $t, \delta \in \R^{+}$. Our goal is to find a prime $p$ such that $p \equiv 1 \bmod 6$ and 
\begin{align}\label{con1}
(t - \delta)p^3 \leqslant m \leqslant (t + \delta)p^3.
\end{align}

Again, if such $p$ is guaranteed, then we let $x =  (p-1)/6$. Hence, when $p$ is sufficiently large, Condition 1 is satisfied if
\begin{align*}
L - \frac{3}{8}p^3 < f(\ell) + f(r) < L.
\end{align*}

Recalling the range of $L$, we must choose $r$ and $\ell$ such that
\begin{align}\label{range1}
\left( t - 2\delta - \frac{1}{2} \right) p^3 < f(\ell) + f(r) < \left( t + \delta - \frac{1}{8} \right)p^3.
\end{align}

If $r, \ell \in ((\alpha - \gamma)p, (\alpha + \gamma)p)$ for some $\alpha, \gamma > 0$, then combining with $\eqref{range1}$, we may assume
\begin{align}\label{Con1part1}
\left( \frac{t}{5} - \frac{2\delta}{5} - \frac{1}{10} \right)^{1/3} + \gamma<\alpha < \left( \frac{t}{5} + \frac{\delta}{5} - \frac{1}{40} \right)^{1/3} - \gamma.
\end{align}

Now we address Condition 2. To satisfy this condition, we must find positive integers $\ell, r$ such that
\[
f(\ell) + f(r) \equiv L \bmod 9p.
\]
Using the Chinese Remainder Theorem, it suffices to consider the congruence equations
\begin{align*}
f(\ell) + f(r) \equiv L \bmod 9 \text{  and } f(\ell) + f(r) \equiv L \bmod p, 
\end{align*}
where the first congruence can be rewritten as
\begin{align}\label{findl0r0}
-2\ell^3 - 2r^3 + 2\ell^2 + 2r^2 + \ell + r \equiv L \bmod 9. 
\end{align}

We then consider $\ell_0, r_0 \in \Z/9\Z$ such that $\ell \equiv \ell_0 \bmod 9$ and $r \equiv r_0 \bmod 9$. We divide \eqref{findl0r0} into 9 cases: if $L \equiv 0 \bmod 9$, let $\ell_0 = r_0 = 0$. If $L \equiv 1 \bmod 9$, let $\ell_0 = 1, r_0 = 0$. If $L \equiv 2 \bmod 9$, let $\ell_0 = r_0 = 1$. If $L \equiv 3 \bmod 9$, let $\ell_0 = 8, r_0 = 0$. If $L \equiv 4 \bmod 9$, let $\ell_0 = 1, r_0 = 2$. If $L \equiv 5 \bmod 9$, let $\ell_0 =1$ and $r_0 = 7$. If $L \equiv 6 \bmod 9$, let $\ell_0 = r_0 = 3$. If $L \equiv 7 \bmod 9$, let $\ell_0 = 4, r_0 = 0$. If $L \equiv 8 \bmod 9$, let $\ell_0 = r_0 = 8$. In all cases, the selected values $\ell_0$ and $r_0$ satisfy \eqref{findl0r0}. Therefore, the original problem reduces to finding positive integers $\ell, r$ such that
\begin{align*}
\begin{split}
r, &\ell \in ((\alpha - \gamma)p,(\alpha + \gamma)p),\\
&\ell \equiv \ell_0 \bmod 9, r \equiv r_0 \bmod 9, \text{ and } f(\ell) + f(r) \equiv L \bmod p,
\end{split}
\end{align*}
where $\alpha$ satisfies \eqref{Con1part1}. We will again use the set up involving Lehmer points to tackle this problem. Let $\mathcal{C}$ be the algebraic curve over $\F_p$ given by
\begin{align*}
f(x_1) + f(x_2) - L = 0.
\end{align*}
 Define the vectors
\begin{align}\label{new vectors}
\textbf{x}& := \begin{bmatrix}
           x_1 \\
           x_2 \\
         \end{bmatrix},
    \textbf{a} := \begin{bmatrix}
           9 \\
           9 \\
         \end{bmatrix},
    \textbf{b} := \begin{bmatrix}
            \ell_0\\
            r_0\\
         \end{bmatrix},
    \textbf{t}:= \begin{bmatrix}
         t_1 \\
          t_2 \\
         \end{bmatrix},
         \textbf{g}:= \begin{bmatrix}
         g_1 \\
          g_2 \\
         \end{bmatrix}. 
\end{align}

With this setup, we have the analogous lemma to Lemma \ref{lem: ZC theorem}.
\begin{lem}[\cite{DoZa2024}, Lemma 9.3]\label{lem: ZC theorem 2}
Let $p$ be a prime, and let $\mathcal{C}$, $\textnormal{\textbf{a}, \textbf{b}, \textbf{g}, \textbf{t}}$ be as defined in \eqref{new vectors}. We define $F(p, \mathcal{C}, \textbf{a}, \textbf{b}, \textbf{g}, \textbf{t})$ as in \eqref{Lehmer Distribution function}. Then, for $p \geqslant 10^{10}$,
\begin{align*}
F(p, \mathcal{C}, \textbf{a}, \textbf{b}, \textbf{g}, \textbf{t}) = \frac{4g_1g_2}{81}(p+1) + \mathcal{E},
\end{align*}
where  $|\mathcal{E}| \leqslant 6.0003\sqrt{p}\log^2 p$.
\end{lem}

Now, choose $t = 0.55, \gamma = 0.207$, and $\delta = 0.3$. Then we choose $\alpha =0.207$, which satisfies the desired range given in \eqref{Con1part1}. Apply Lemma \ref{lem: ZC theorem 2} with $t_1=t_2=\alpha$ and $g_1=g_2=\gamma$. In order to have $F(p,\mathcal{C},\textnormal{\textbf{a}, \textbf{b}; \textbf{g}, \textbf{t}}) > 0$, it suffices to achieve
\begin{align}\label{Prime Requirement 2}
\gamma^2(p+1)>\frac{81}{4} (6.0003\sqrt{p}\log^2p).
\end{align}

For all $m\geqslant \num{2.9446e33}$, the corresponding prime $p> \num{1.5131e11}$, thereby fulfilling \eqref{Prime Requirement 2}. By restricting $p$ to the sub-interval $[1.26\sqrt[3]{m}, 1.275\sqrt[3]{m}]$, we ensure that Condition 1 holds for all $m\geqslant 16959$. To establish the existence of such a prime $p$ in this sub-interval, we again invoke Theorem 1 in \cite{ramare1996} with $\varepsilon = 0.004560$, which yields
\begin{align*}
    \max_{1\leqslant y\leqslant 1.275\sqrt[3]{m}}\bigg|\theta(y;6,1)-\frac{y}{2}\bigg|\leqslant 0.002907\sqrt[3]{m}.
\end{align*}
Taking $y=1.275\sqrt[3]{m}$, we have
\begin{align}\label{upper chebyshev 2}
    \bigg|\theta(1.275\sqrt[3]{m};6,1)-\frac{1.275\sqrt[3]{m}}{2}\bigg|\leqslant 0.002907\sqrt[3]{m}.
\end{align}

Similarly, take $y=1.26\sqrt[3]{m}$, we have
\begin{align}\label{lower chebyshev 2}
    \bigg|\theta(1.26\sqrt[3]{m};6,1)-\frac{1.26\sqrt[3]{m}}{2}\bigg|\leqslant 0.0028728\sqrt[3]{m}.
\end{align}

Combining \eqref{upper chebyshev 2} and \eqref{lower chebyshev 2}, we obtain
\begin{align*}
    \sum\limits_{\substack{1.26\sqrt[3]{m}<p\leqslant 1.275\sqrt[3]{m} \\ p \equiv 1 \bmod 6}}\log p&=\theta(1.275\sqrt[3]{m};6,1)-\theta(1.26\sqrt[3]{m};6,1)\notag\\
    &\geqslant (0.6375\sqrt[3]{m}-0.002907\sqrt[3]{m})-(0.63\sqrt[3]{m}+0.0028728\sqrt[3]{m})\notag\\
    & = 0.0017202\sqrt[3]{m} > 0.
\end{align*}

Thus, the existence of such prime $p$ is guaranteed in the sub-interval $[1.26\sqrt[3]{m}, 1.275\sqrt[3]{m}\ ]$, which completes the proof of Lemma \ref{lem: 6 D + 2 I}. \qed

\section{Proof of Theorem \ref{thm: main theorem}}\label{sec: proof of main theorem}
To complete the proof, we make use of the four lemmas in Section \ref{sec: 6+2 lemmas}. Observe that from the pairs $(5,6), (4,7),(3,8)$, and $(2,9)$, we can extract two icosahedral numbers and six dodecahedral numbers. Therefore, applying Lemma \ref{lem: 6 D + 2 I}, it suffices to show that these pairs are essential pairs for $m<\num{2.9446e33}$. Similarly, we may apply Lemma \ref{lem: 8D} to the pairs $(1,11)$ and $(0,22)$, Lemma \ref{lem: 6 I + 2 D} to $(6,5),(7,4),(8,3),(9,2)$, and Lemma \ref{lem: 8I} to $(10,1)$ and $(15,0)$. By the results in \cite{DoZa2024}, $(15,0)$ and $(0,22)$ are confirmed to be essential pairs, so we only need to focus on the other 10 pairs.

For brevity, we will focus on the pairs $(5,6), (4,7),(3,8)$, and $(2,9)$, and the proofs for the other pairs follow similarly. To show that these four pairs are essential, we make use of the following steps. \\

Let $(a,b)$ be the pair that we want to check, $i=0$, and $m_i = m$. 
\subsection*{Step 1.}  Let $n_i$ be the integer such that $g(n_i)< m_i< g(n_i+1)$. If $m_i\in (g(n_i),g(n_i)+10000)$, then let $m_{i+1}=m_i-g(n_i-1)$. Otherwise, let $m_{i+1}=m_i-g(n_i)$.
\subsection*{Step 2.} Take $m_i - m_{i+1}$ to be the first term in the sum, and reduce the problem to writing $m_{i+1}$ as a sum of one fewer dodesahedral numbers. 
\subsection*{Step 3.} Increase $i$ by one and repeat Step 1 to 3.\\

Apply the above steps to a positive integer $m_0 < \num{2.9446e33}$. If there exists some $n_0\in \mathbb{N}$ such that $m_0\in [g(n_0)+10000, g(n_0+1))$, then let $m_1=m_0-g(n_0)$. We deduce that 
\begin{align}
       & m_1 < g(n_0+1)-g(n_0)= \frac{27}{2}n_0^2+\frac{9}{2}n_0+1.\label{m_1 bound}
\end{align}

Since $m_0<\num{2.9446e33}$, we have
\[
g(n_0) = \frac{9}{2}n_0^3-\frac{9}{2}n_0^2+n_0<\num{2.9446e33},
\]
which implies that $n_0<\num{8.682e10}.$ Substituting this into \eqref{m_1 bound}, we obtain $m_1<\num{1.018e23}$. When $m_0\in (g(n_0),g(n_0)+10000)$, by taking $m_1 = m_0-g(n_0-1)$, we have
\begin{align*}
m_1 < g(n_0)+10000-g(n_0-1) = \frac{27}{2}n_0^2-\frac{45}{2}n_0+10010.
\end{align*}

Again, since $m_0<\num{2.9446e33}$, we conclude that $m_1<\num{1.018e23}$ as well. Therefore, in both cases, it suffices to write $10000\leqslant m_1\leqslant \num{1.018e23}$ as a sum of $a$ icosahedral numbers and $(b-1)$ dodecahedral numbers.

Note that in any round, we can switch from dodecahedral numbers to icosahedral numbers, that is, reducing the problem to writing $m_i$ as the sum of one fewer icosahedral numbers. Then we will have
\begin{align*}
m_i&<\max\left\{f(n_i+1)-f(n_i), f(n_i)-f(n_i-1)+10000\right\} \\
&= \max\left\{\frac{15}{2}n_i^2+\frac{5}{2}n_i+1,\frac{15}{2}n_i^2-\frac{25}{2}n_i+10006 \right\}.
\end{align*}

Applying the steps accordingly to pairs $(5,6),(4,7),(3,8),$ and $(2,9)$, the original problem for any of these pairs can be reduced to writing $10000\leqslant m_5\leqslant \num{1.666e6}$ as a sum of two icosahedral numbers and four dodecahedral numbers.

For the pair $(10,1)$, we apply the above steps five times and reduce the problem to representing an integer $10000\leqslant m_5\leqslant \num{2.13e6}$ as a sum of at most five icosahedral numbers and one dodecahedral number. For the pairs $(6, 5), (7, 4), (8, 3), (9, 2)$, five applications of these steps reduce the problem to writing an integer $m_5$ such that $10000\leqslant m_5\leqslant \num{5.74e5}$ as a sum of at most four icosahedral numbers and two dodecahedral numbers. The remaining pair $(1,11)$ requires representing an integer $m_6$ with $\num{12000} \leqslant m_6 \leqslant \num{2.1e5}$ as a sum of at most one icosahedral number and five dodecahedral numbers. We complete the proof for all cases using the corresponding algorithms in the ``Algorithm II" folder of \cite{github}.
\qed

\printbibliography

@article{bombieri,
 URL = {http://www.jstor.org/stable/2373048},
 author = {Bombieri, E.},
 journal = {American Journal of Mathematics},
 number = {1},
 pages = {71--105},
 publisher = {Johns Hopkins University Press},
 title = {On Exponential Sums in Finite Fields},
 urldate = {2024-05-23},
 volume = {88},
 year = {1966}
}

@article{pollock,
 author = {Pollock, S. F.},
 journal = {Proceedings of the Roy Society of London},
 title = {On the extension of the principle of Fermat’s theorem of the polygonal numbers to the higher
orders of series whose ultimate differences are constant.},
 year = {1850},
volume = {5},
pages = {922–924}

}

@article{ramare1996,
  title={Primes in Arithmetic Progressions},
  author={Ramaré, O. and Rumely, R.},
  journal={Mathematics of Computation},
  volume={65},
  number={213},
  pages={397--425},
  year={1996},
  publisher={American Mathematical Society}
}

@article{DoZa2024,
  title={Representations as Sums of Icosahedral and Dodecahedral Numbers: Proof of Pollock’s Conjectures},
  author={Basak, D. and Dong, A. and Saettone,  K and Zaharescu, A.},
  journal={International Mathematics Research Notices},
  year={2025},
 volume = {2025}, 
}

@article{cobeli,
author = {Cobeli, C. and Zaharescu, A.},
year = {2001},
pages = {301-307},
title = {Generalization of a problem of Lehmer},
volume = {104},
journal = {Manuscripta Mathematica},
doi = {10.1007/s002290170028}
}

@article{Brady_2015,
   title={Sums of seven octahedral numbers},
   volume={93},
   url={http://dx.doi.org/10.1112/jlms/jdv061},
   DOI={10.1112/jlms/jdv061},
   number={1},
   journal={Journal of the London Mathematical Society},
   publisher={Wiley},
   author={Brady, Z. E.},
   year={2015},
   pages={244–272} }

@article{hilbert1909,
    author = {Hilbert, D.},
    title = {Beweis f\"ur {D}arstellbarkeit der ganzen {Z}ahlen durch eine feste {A}nzahl nter {P}otenzen ({W}aringsche {P}roblem)},
    journal = {Mathematische Annalen},
    year = {1909},
    volume= {67},
    pages={281--300}
}

@article{Wieferich1908BeweisDS,
  title={Beweis des {S}atzes, da{\ss} sich eine jede ganze {Z}ahl als {S}umme von h{\"o}chstens neun positiven {K}uben darstellen l{\"a}{\ss}t},
  author={Wieferich, A.},
    journal = {Mathematische Annalen},
  year={1908},
  volume={66},
  pages={95-101},
  url={https://api.semanticscholar.org/CorpusID:121386035}
}

@article{Kempner1912BemerkungenZW,
  title={Bemerkungen zum {W}aringschen {P}roblem},
  author={Kempner, A. J.},
    journal = {Mathematische Annalen}, 
  year={1912},
  volume={72},
  pages={387-399},
  url={https://api.semanticscholar.org/CorpusID:120101223}
}

@article{balasubramanian,
    AUTHOR = {Balasubramanian, R.},
     TITLE = {On {W}aring's problem: {$g(4)\leq 20$}},
   JOURNAL = {Hardy-Ramanujan Journal},
    VOLUME = {8},
      YEAR = {1985},
     PAGES = {1--40}
}

@article{deshouillers,
    AUTHOR = {Deshouillers, J. M. and Dress, F.},
     TITLE = {Sums of {$19$} biquadrates: on the representation of large
              integers},
   JOURNAL = {Annali della Scuola Normale Superiore di Pisa. Classe di
              Scienze. Serie IV},
    VOLUME = {19},
      YEAR = {1992},
    NUMBER = {1},
     PAGES = {113--153}
}

@article{chen1964,
    author = { Chen, J.R.},
    title = {Waring’s problem for $g(5) = 37$},
    journal = {Scientia Sinica. Zhongguo Kexue},
    year = {1964},
    volume = {13},
    pages = {1547--1568}
}

@article{pillai1940,
    author = { Pillai, S. S.},
    title = {On {W}aring’s problem: $g(6) = 73$},
    journal = {Proceedings of the Indian Academy of Sciences},
    year = {1940},
    volume={12},
    pages={30--40}
}

@online{github,
  author = {Dong, A.},
  title = {Algorithm for Admissible Pairs},
  year = {2025},
  publisher = {GitHub},
  note = {Available at: \href{https://github.com/Anjisweety2/Algorithm-for-Admissible-Pairs.git}{Algorithm for Admissible Pairs}}
}

\end{document}